\documentclass[12pt]{amsart}
\usepackage{geometry, amsmath, amsthm,amsfonts,amssymb,stmaryrd, mathrsfs}
\usepackage{fancyhdr}
\usepackage{hyperref}

\usepackage[all]{xy}

\newtheorem*{thm}{Theorem}

\newtheorem{lemma}[subsection]{Lemma}

\newtheorem{proposition}[subsection]{Proposition}
\theoremstyle{definition}

\newtheorem{example}[subsection]{Example}
\newtheorem{remark}[subsection]{Remark}

\numberwithin{equation}{subsection}

\def\calH{\mathcal{H}}

\def\calO{\mathcal{O}}

\def\calV{\mathcal{V}}
\def\calW{\mathcal{W}}

\def\gothC{\mathfrak{C}}
\def\gothE{\mathfrak{E}}

\def\AAA{\mathbb{A}}

\def\CC{\mathbb{C}}
\def\GG{\mathbb{G}}
\def\NN{\mathbb{N}}
\def\QQ{\mathbb{Q}}
\def\RR{\mathbb{R}}

\def\ZZ{\mathbb{Z}}

\def\bff{\mathbf{f}}

\def\rmR{\mathrm{R}}

\def\ttv{\mathtt{v}}

\DeclareMathOperator{\Gal}{Gal}

\DeclareMathOperator{\Hom}{Hom}
\DeclareMathOperator{\HT}{HT}

\DeclareMathOperator{\Spec}{Spec}
\DeclareMathOperator{\Spm}{Spm}
\DeclareMathOperator{\Sym}{Sym}

\newcommand{\can}{\mathrm{can}}
\newcommand{\cont}{\mathrm{cont}}
\newcommand{\cycl}{\mathrm{cycl}}

\newcommand{\Fil}{\mathrm{Fil}}

\newcommand{\id}{\mathrm{id}}

\newcommand{\pr}{\mathrm{pr}}
\newcommand{\Qp}{\QQ_p}
\newcommand{\rig}{\mathrm{rig}}
\newcommand{\Tate}{\mathrm{Tate}}

\newcommand{\WT}{\mathrm{WT}}
\newcommand{\Zp}{\ZZ_p}

\fancypagestyle{plain}{
\fancyhf{} % clear all header and footer fields

}
\begin{document}

\title{Gauss--Manin connections for $p$-adic families of nearly overconvergent modular forms}
%En Francais: Connexions de Gauss--Manin pour les families $p$-adiques de formes modulaires quasi-surconvergentes
\author{Robert Harron}
\author{Liang Xiao}

\date{July 14, 2014 (originally: August 7, 2013)}

%En Francais: Connexions de Gauss--Manin, Formes modulaires quasi-surconvergents, courbes de Hecke, Familles $p$-adiques de formes modulaires

\begin{abstract}
We interpolate the Gauss--Manin connection in $p$-adic families of nearly overconvergent modular forms.
This gives a family of Maass--Shimura type differential operators from the space of nearly overconvergent modular forms of type $r$ to the space of nearly overconvergent modular forms of type $r+1$ with $p$-adic weight shifted by $2$.
Our construction is purely geometric, using Andreatta--Iovita--Stevens and Pilloni's geometric construction of eigencurves, and should thus generalize to higher rank groups.

\vspace*{12pt}
\noindent\textsc{R\'esum\'e.} Nous obtenons l'interpolation de la connexion de Gauss--Manin en familles $p$-adiques de formes modulaires quasi-surconvergentes. Ceci donne une famille d'op\'erateurs diff\'erentiels \`a la Maass--Shimura qui envoie l'espace de formes modulaires quasi-surconvergentes de type $r$ dans celui de formes modulaires quasi-surconvergentes de type $r+1$ et de poids $p$-adique augment\'e par $2$. Notre m\'ethode est purement g\'eom\'etrique, utlise les constructions g\'eom\'etriques des courbes de Hecke dues \`a Andreatta--Iovita--Stevens et Pilloni, et devrait donc se g\'en\'eraliser aux groupes de rang sup\'erieur.
\end{abstract}


\maketitle

\setcounter{tocdepth}{1}
\tableofcontents

\section{Introduction}
In this article, we seek to combine two important tools in arithmetic: the nearly holomorphic modular forms of Shimura and the $p$-adic families of modular forms of Hida and Coleman--Mazur. The former are an integral part of Shimura's study of the algebraicity of values of automorphic $L$-functions while the latter have become a ubiquitous tool in number theory with recent applications including the proof of Serre's conjecture, the proof of the Fontaine--Mazur--Langlands conjecture for $GL(2)$, and the construction of automorphic Galois representations to name a few. The $p$-adic theory of nearly holomorphic modular forms has recently emerged from the work of Darmon--Rotger \cite{DarmonRotger} and Skinner--Urban (see \cite{Urban2}) with applications to the Birch--Swinnerton-Dyer and Bloch--Kato conjectures for elliptic curves and to the theory of Stark--Heegner points. Since it is sufficient for their purposes, these authors rely on $q$-expansions in their definition of nearly overconvergent modular forms. Our approach, via the work of Andreatta--Iovita--Stevens \cite{AIS} and Pilloni \cite{pilloni}, is geometric: we give a geometric definition of nearly overconvergent modular forms of arbitrary $p$-adic weight and construct $p$-adic families of Gauss--Manin connections varying over the weight space. This provides a robust theory that is amenable to generalization.

In order to state our main theorem, let us introduce the following notation: for an affinoid algebra $A$, a continuous character $\chi_A: \ZZ_p^\times \to A^\times$, a non-negative integer $r$, and a positive rational number $v$, let $M_r^{\dagger,v}(\chi_A)$ denote the Banach $A$-module of nearly overconvergent modular forms of type $r$, weight $\chi_A$, and radius of convergence $v$. Our main result is then the following.

\begin{thm}
For sufficiently small $v$, we have an $A$-linear map
\[
\nabla^{\chi_A, r}: M_r^{\dagger,v}(\chi_A) \to M_{r+1}^{\dagger,v}(\chi_A\chi_\cycl^2),
\]
such that for each point $x \in \Spm(A)$ that corresponds to a classical weight, the specialization of $\nabla^{\chi_A,r}$ to $x$ restricts to the classical Gauss--Manin connection.
\end{thm}

This theorem is also proved in \cite{urban} using a subtle argument heavily dependent on the $q$-expansion principle.  Our construction (which was worked out independently) is purely geometric and hence can be easily generalized to similar situations whenever there is a geometric construction like the ones given by Andreatta--Iovita--Stevens \cite{AIS} and Pilloni \cite{pilloni}.

We now briefly sketch our construction of $\nabla^{\chi_A,r}$.
Let $X$ be the modular curve over $\QQ_p$, with cuspidal subscheme $C$, and let $E$ be the universal generalized elliptic curve over $X$.  Let $\omega \subseteq \calH$ be the sheaf of relative $1$-differentials and the relative de Rham cohomology sheaf of $E \to X$, respectively.
For integers $0 \leq r <k$, the Gauss--Manin connection gives rise to a connection $\nabla^{k,r}: \omega^{k-r} \otimes \Sym^r \calH \to \omega^{k+1-r} \otimes \Sym^{r+1} \calH$.
We choose a splitting of the Hodge filtration $\calH = \omega \oplus \omega^{-1}$.
Then, the Gauss--Manin connection $\nabla: \calH \to \calH \otimes \Omega_X^1(\log C)$ can be reconstructed from a 
differential operator $\partial: \omega \to \omega \otimes \Omega^1_X(\log C)$.
We interpret $\partial$ as a connection on the associated principal $\GG_m$-bundle for $\omega$, which we further restrict to a connection on the principal bundle for the rigid balls inside $\GG_m$ appearing in the construction of Andreatta--Iovita--Stevens \cite{AIS} and Pilloni \cite{pilloni}.  Then, we obtain our family version of $\partial$ by applying the formal dictionary between principal bundles and vector bundles associated to the representation of the monodromy group.
The family of Gauss--Manin connections can be then reconstructed from the family version of $\partial$ as in Katz's paper \cite{katz}.

One of the applications of our result is the construction of the $p$-adic Rankin--Selberg convolution over the product of two Coleman--Mazur eigencurves, which was in fact the original motivation of this paper.  We now refer to \cite{urban} for this construction.

We also note that our method of construction should be quite general. For instance, we expect a similar construction for the case of Siegel modular forms making use of the work of Andreatta--Iovita--Pilloni \cite{AIP} (or its Hilbert--Siegel and PEL variants).

\subsection*{Acknowledgments}
We thank Tsao-Hsien Chen, Matthew Emerton,  Madhav Nori, and Eric Urban for useful discussions. We also thank the referee for pointing out some typos and inaccuracies.

\section{A geometric construction of nearly overconvergent modular forms}
We first review the geometric construction of families of nearly overconvergent modular forms. 
This approach essentially follows from the work of Andreatta, Iovita, and Stevens \cite{AIS} and, independently, Pilloni \cite{pilloni}.  The advantage of this construction is that nearly overconvergent modular forms are on the nose sections of certain vector bundles.

\subsection{Weight spaces}
We fix a prime number $p$.  Let $\CC_p$ denote the completion of a fixed algebraic closure of $\QQ_p$.  It is equipped with a valuation $\ttv: \CC_p^\times \to \RR$ normalized so that $\ttv(p) = 1$.  Let $|\cdot| = p^{-\ttv(\cdot)}$. 
Put $q = 4$ if $p =2$ and $q =p$ if $p \neq 2$.\footnote{The confusion with $q = e^{2 \pi iz}$ later should be minimal.} 
 We can then write $\ZZ_p^\times = (\ZZ/q\ZZ)^\times \times (1+ q\ZZ_p)^\times$.

We put $\Gamma = \Gal(\QQ(\mu_{p^\infty})/\QQ)$, where $\mu_{p^\infty}$ denotes the collection of all $p$-power roots of unity.  We use the \emph{cyclotomic character} $\chi_\cycl: \Gamma \to \ZZ_p^\times$ to identify these two groups.

Let $\calW$ denote the rigid analytic space over $\Qp$ whose $K$-points are $\calW(K) = \Hom_\cont(\Gamma, K^\times)$ for any complete field extension $K$ of $\QQ_p$.  We have an isomorphism
\[
\xymatrix@R=0pt@C=40pt{
\calW \ar[r]^-{\cong} &
\widehat{(\ZZ/q\ZZ)^\times} \times B(1, 1^-)\\
\chi \ar@{|->}[r] &\big(\ \chi|_{(\ZZ/q\ZZ)^\times},\ \chi(\exp(q))\ \big),
}
\]
where $B(1, 1^-)$ denotes the open unit ball of radius 1 centered at $1$.

A \emph{classical character} of $\Gamma$ is a character of the form $\epsilon\chi_\cycl^k :  \Gamma \to K^\times$ with $k \geq 2$ an integer and $\epsilon$ a \emph{finite order} continuous character of $\Gamma$.  We will use $(\epsilon, k)$ to denote such a character and also the associated point on $\calW$.  The \emph{$p$-conductor} of such $\epsilon$ is $p^n$, where $n = n(\epsilon)$ is the maximal \emph{positive integer} such that $\epsilon$ factors through $(\ZZ/p^{n-1}q\ZZ)^\times$. (In particular, if $\epsilon$ is trivial, our convention says that $n(\epsilon)=1$.)  We remark that the set of classical characters is Zariski dense in the weight space $\calW$.

Let $A$ be an affinoid algebra over $\QQ_p$.  (A good example to keep in mind is the ring of analytic functions on an affinoid subdomain of the weight space $\calW$.)
For a continuous character $\chi_A: \ZZ_p^\times \to A^\times$, its \emph{weight} is defined to be 
\[
\WT(\chi_A) := \lim_{a \to 1} \frac{\log(\chi_A(a))}{\log a} \in A.
\]
We sometimes view it as a function on $\Spm(A)$.  When $\chi_A = (\epsilon,k)$ is a classical character, its weight is simply $k$.

\subsection{Nearly algebraic modular forms}\label{sec:nearlyalgebraic}
Let $X\to \Spec \ZZ_p$ denote the compactified modular curve with hyperspecial level at $p$ over $\Spec \ZZ_p$.
We always take the tame level subgroup to be sufficiently small so that $X$ is a fine moduli space of generalized elliptic curves.
We fix such a tame level throughout the paper.
Let $C$ denote the cusp subscheme.  Let $\pi: E \to X$ be the universal generalized elliptic curve.
Put $\tilde C = \pi^{-1}(C)$; it is a divisor of $E$ with simple normal crossings.
Let $e: X \to E$ be the unit section and put $\omega: =e^*\Omega_{E/X}^1$. The Kodaira--Spencer isomorphism gives
\[
\Omega^1:=\Omega^1_{X/\ZZ_p}(\log C)\cong
\omega^{\otimes 2}.
\]
We use $\calH$ to denote the relative de Rham cohomology $\rmR^1\pi_*(\Omega^\bullet_{E/X}(\log \tilde C))$.\footnote{Here, $\Omega^1_{E/X}(\log \tilde C) = \Omega^1_{E/\Zp}(\log \tilde C) / \pi^*\Omega^1_{X/\Zp}(\log C)$.}

For integers $k \geq 2$ and $r \in [0,\frac k2-1]$, following Urban \cite{urban}, we define the \emph{nearly algebraic modular forms} of weight $k$ and type $r$ to be 
\[
M_{k, r} := H^0(X_{\Qp}, \omega^{k-r} \otimes \Sym^r \calH).
\]
When $r=0$, this recovers the space of usual modular forms $M_k$, which we call \emph{classical modular forms} in this paper.  These nearly algebraic modular forms are closely related to the nearly holomorphic modular forms of Shimura.  We refer to \cite{urban} for a discussion of the relation between the two.

\subsection{The ball fibration}
We now review a construction given in \cite{AIS} and \cite{pilloni}.
We put $T = \Spec(\oplus_{n \in \ZZ_{\geq 0}} \omega^{-n})$ and $T^\times = \Spec(\oplus_{n \in \ZZ} \omega^n)$.  The space $T$ may be viewed as the physical line bundle over $X$ associated to $\omega$ and $T^\times$ as the $\GG_m$-torsor that defines the line bundle $T$.  Let $X_\rig$ denote the analytification of $X_{\Qp}$ (as a rigid analytic space in the sense of Tate).  Let $T_\rig$ be the analytification of $T_{\QQ_p}$.  Let $T_\rig^\times$ be the analytification of $T^\times_{\Qp}$; it is a $\GG_{m, \rig}$-torsor over $X_\rig$.

There is a continuous map $\deg: X_\rig \to [0,1]$, given by the valuation of the truncated Hasse invariant.  
More precisely, we fix a lift $\tilde h \in H^0(X, \omega^{p-1})$ of the Hasse invariant; for $x \in X_\rig$, $\deg(x) = \min\{ \ttv(\tilde h(x)), 1\}$.\footnote{The notation deg refers to the degree of the corresponding finite flat $p$-group scheme in the sense of \cite{fargues}.}
For any $v \in [0,1]\cap p^\QQ$, $X(v) :=\{x \in X_\rig \;|\; \deg(x) \leq v\}$ is a subdomain of $X_\rig$.  In particular, $X(0)$ is the tube of the ordinary locus of the special fiber of $X$.

For $n \in \ZZ_{\geq1}$ and $v \in [0, \frac 1{p^{n-2}(p+1)})$, the universal generalized elliptic curve $E$ over $X(v)$ has a canonical subgroup $C_n$ of order $p^n$.  
For each closed point $x \in X(v)(K)$, we use $C_{n,x}$ to denote the corresponding canonical subgroup of the (generalized) elliptic curve $E_x$ at the point $x$.  We may choose formal models $\gothC_{n,x}$ and $\gothE_x$ of both objects over $\calO_K$.
We have a Hodge--Tate map
\begin{equation}
\label{E:Hodge--Tate-map}
\HT: \gothC_{n,x}^D(\calO_{\CC_p}) \to e^*(\Omega^1_{\gothC_{n,x} /\calO_K}) \otimes_{\calO_K} \calO_{\CC_p},
\end{equation}
where $\gothC_{n,x}^D$ denotes the Cartier dual of $\gothC_{n,x}$. 

Using the Hodge--Tate map,\footnote{Technically speaking, just knowing the description of Hodge--Tate map over closed points is not enough to prove the existence of $F_n$.  Pilloni \cite{pilloni} needed a version of the Hodge--Tate map over a formal scheme, just as in \cite{AIS}.  However, this subtlety does not matter to our discussion, so we ignore it.} Pilloni \cite[Th\'eor\`eme~3.2]{pilloni} showed that there exists an open rigid subdomain $F_n \hookrightarrow T_\rig \times_{X(v)} C_n^D$ such that, for any closed point $x \in X(v)(K)$, we have
\[
F_n|_x(\CC_p) = \big\{(y, \omega) \in \gothC_{n,x}^D(\CC_p) \times \big(e^*\Omega_{\gothE_x/\calO_K}^1 \otimes_{\calO_K}  \calO_{\CC_p}\big) \text{ with } \HT(y) = \omega|_{e^*\Omega^1_{\gothC_{n,x} /\calO_K}\otimes_{\calO_K} \calO_{\CC_p}} \big\}
\]
where $F_n|_x$ denotes the fiber of $F_n$ above $x \in X(v)(K)$.

Let $(C_n^D)^\times$ denote the union of the connected components of $C_n^D$ formed by local generators of $C_n^D$.
We put $F_n^\times = F_n \times_{C_n^D} (C_n^D)^\times$. 
When $v<\frac{p-1}{p^n}$, Pilloni \cite[Proposition~3.5]{pilloni} proved that the natural morphism $F_n^\times \to F_n \to T_\rig$ factors through $T_\rig^\times$ and is an open immersion.  In a more explicit form, over $x \in X(v)(\CC_p)$, we may choose a generator of $\omega$ over $X$ and hence identify $T_\rig$ with $\AAA^1_{\CC_p, \rig}$.  Then 
\[
F_n^\times|_x(\CC_p) =  \coprod_{m \in (\ZZ/p^n\ZZ)^\times}
x_m h^{1/(p-1)} + p^n h^{-\frac{p^n-1}{p-1}}\calO_{\CC_p} \subset \CC_p,
\]
where $h \in \CC_p$ is some element with $\ttv(h) = \deg(x)$ and $\{x_m\, |\,m \in (\ZZ/p^n\ZZ)^\times\}$ is some fixed set of lifts of $(\ZZ/p^n\ZZ)^\times$ to $\ZZ_p^\times$.

From this, it is easy to see (and is also explained in \cite[\S3.4]{pilloni} and \cite[\S3.2]{AIS} implicitly) that $\pr_n: F_n^\times \to X(v)$ is a torsor for the open subgroup $G_n \subseteq X(v) \times \GG_{m, \rig}$ given by
\begin{equation}
\label{E:G_n}
G_n(\CC_p) = \big\{(x, a) \in X(v)(\CC_p) \times \CC_p^\times \;\big|\; |a-m|\leq p^{-n} p^{p^n\deg(x)/(p-1)} \textrm{for some }m \in \ZZ_p^\times
\big\}.
\end{equation}

\begin{remark}
The following viewpoint of the above construction was communicated to us by Emerton.
The line bundle $\omega$ defines a $\GG_m$-torsor over (the entirety of) $X$; this is why we can consider integer powers of $\omega$ and define modular forms with integer weights.
The dual Tate module of the $p^\infty$-canonical subgroup of $E$ only exists over the ordinary locus $X(0)$; it gives rise to a $\ZZ_p^\times$-torsor over $X(0)$.  This is why Hida can consider $p$-adic families of modular forms parameterized by $\ZZ_p\llbracket \ZZ_p^\times \rrbracket$.
Over the ordinary locus, the Hodge--Tate map is an isomorphism.  It identifies the aforementioned dual Tate module as a $\ZZ_p^\times$-subtorsor of the $\GG_m^\times$-torsor $T^\times$.  As we include some supersingular locus, i.e. we work over $X(v)$, the Hodge--Tate map fails to be an isomorphism.  Instead, we see a torsor $F^\times_n$ for the group $G_n$, which is a group sitting between $\ZZ_p^\times$ and $\GG_{m, \rig}$.
\end{remark}

\subsection{Locally analytic characters}
\label{S:locally analytic characters}
Let $A$ be an affinoid algebra and $\chi_A: \ZZ_p^\times \to A^\times$ a continuous character. 
Then there exists an integer $n \geq 3$ such that $\lambda_{A,n-1}: = \chi_A(\exp(p^{n-1})) \in A$ satisfies $|\lambda_{A, n-1}-1|<p^{-1/(p-1)}$.\footnote{Here, we did not optimize the choice of $n$ for simplicity of the presentation; see \cite[Section~2.1]{pilloni} for a careful discussion of the optimal bound.}
For any such $n$, the character $\chi_A:  \ZZ_p^\times \to A^\times$ extends by continuity to a \emph{locally analytic character}
\[
\xymatrix@R=0pt@C=15pt{
\chi_A^\mathrm{an} = (\chi_A, \lambda_A): & \ZZ_p^\times\cdot (1+p^{n-1}\calO_{\CC_p})
^\times \ar[rr]&& (A \widehat\otimes \CC_p)^\times\\
& x\cdot z \ar@{|->}[rr] && \chi_A(x) \cdot \lambda_{A,{n-1}}^{\log(z)/p^{n-1}},
}
\]
where $x \in \ZZ_p^\times$ and $z \in (1+p^{n-1}\calO_{\CC_p})
^\times$.  In other words, this defines a family of representations of the rigid analytic group
\[
\overline{G}_{n-1} = \coprod_{m \in (\ZZ/p^{n-1}\ZZ)^\times} B(x_m, p^{1-n})
\]
parametrized by $\Spm(A)$,
where $B(x_m, p^{1-n})$ denotes the union of the  closed disks of radius $p^{1-n}$ centered at the chosen representatives $x_m \in \ZZ_p^\times$ as $m$ varies over all classes in $(\ZZ/p^{n-1}\ZZ)^\times$.

\subsection{Overconvergent sheaves}
Keep the notation as above.
Choose $v \in (0, \frac{p-1}{p^{n}}) \cap \QQ$.  
The representation $\chi_A^\mathrm{an}$ of $\overline{G}_{n-1}$ induces a line bundle $\calV$ over $X(v) \times \Spm(A)$ on which $\overline{G}_{n-1} \times X(v)$ acts.
Our assumption on $v$ ensures that $G_n$ is a subgroup of $\overline{G}_{n-1} \times X(v)$ by the description~\eqref{E:G_n}.  We may thus consider the induced action of $G_n$ on $\calV$.

Let $\mathrm{pr}_n \times \id: F^\times_n \times \Spm A \to X(v) \times \Spm A$ denote the map induced by $\mathrm{pr}_n$.
Following \cite[\S5.1]{pilloni}, we define the \emph{modular sheaf} with character $\chi_A$ to be
\[
\boldsymbol \omega^{\chi_A}:= \big((\mathrm{pr}_n \times \id)_* \calO_{F_n^\times \times \Spm A} \otimes_{\calO_{X(v) \times \Spm A}} \calV \big)^{G_n}.
\]
Since $F^\times_n$ is a $G_n$-torsor, $\boldsymbol \omega^{\chi_A}$ is a locally free sheaf over $X(v) \times \Spm(A)$ of rank one.

For any such $\chi_A$, we define an \emph{overconvergent modular form with character $\chi_A$} to be a section of $\boldsymbol \omega^{\chi_A}$ over $X(v) \times \Spm(A)$ for some $v$ as above.  We put
\begin{equation}
\label{E:ocvgt modular forms}
M^{\dagger,v}(\chi_A) = \Gamma\big(X(v) \times \Spm(A), \boldsymbol \omega^{\chi_A}\big)\quad \textrm{ and }\quad M^{\dagger}(\chi_A) = \varinjlim_v M^{\dagger,v}(\chi_A).
\end{equation}
More generally, for $r \in \ZZ_{\geq0}$, we define the space of \emph{nearly overconvergent modular forms} to be
\[
M^{\dagger, v}_r(\chi_A) = 
\Gamma\big(X(v) \times \Spm(A), \boldsymbol \omega^{\chi_A\chi_\cycl^{-r}} \otimes \Sym^r \calH\big)\quad \textrm{ and }\quad M_r^{\dagger}(\chi_A) = \varinjlim_v  M_r^{\dagger,v}(\chi_A).
\]
We have natural inclusions $M^{\dagger(, v)}_{r'}(\chi_A) \hookrightarrow M^{\dagger(, v)}_r(\chi_A)$ if $r' \leq r$, and 
$M^{\dagger(, v)}_0(\chi_A) = M^{\dagger(, v)}(\chi_A)$.

When $\chi_A = \chi_\cycl^k: \ZZ_p^\times \to K^\times$ is  the $k$-th power of the cyclotomic character, $\boldsymbol \omega^{\chi_\cycl^k}$ is the restriction of the usual modular sheaf $\omega^k|_{X(v)}$ for  $v \in (0,1)$ sufficiently close to $0$ (\cite[Proposition~3.6]{pilloni}).  In this case, we put
\[
M^{\dagger, v}_k = M^{\dagger, v}(\chi_A)\quad \textrm{ and }\quad M^{\dagger}_k = M^{\dagger}(\chi_A).
\]
These are the \emph{overconvergent modular forms of weight $k$} in the usual sense; they contain the classical modular forms $M_k$ as a subspace.

Similarly, for an integer $r\geq 0$, we define the space of \emph{nearly overconvergent modular forms} of weight $k$ to be
\[
M^{\dagger, v}_{k,r} = M^{\dagger, v}_r(\chi_A)\quad \textrm{ and }\quad M^{\dagger}_{k,r} = M^{\dagger}_r(\chi_A).
\]
These contain the space of nearly algebraic modular forms $M_{k,r}$ as a subspace when $k \geq 2r+2$.

It is clear from the construction that, when $\chi = (\epsilon,k)$ is a classical character, our construction agrees with that of \cite{urban}.

\begin{remark}
One can define Hecke actions on nearly overconvergent modular forms as in \cite[\S4]{pilloni}.
We refer to {\it loc.\ cit.} for details.
\end{remark}

\section{The Gauss--Manin connection in a $p$-adic family}
We now give the construction of the Gauss--Manin connection in a $p$-adic family over the weight space.

\subsection{Gauss--Manin connections}
\label{S:GM connection}
The relative de Rham cohomology over the modular curve $X$ fits into a canonical exact sequence
\begin{equation}
\label{E:Hodge filtration}
0 \to \omega \to \calH \to \omega^{-1} \to 0.
\end{equation}
Equivalently, $\calH$ admits a Hodge filtration given as follows: $\Fil^i\calH = 0$ for $i >1$; $\Fil^1\calH = \omega$; and $\Fil^i\calH = \calH$ for $i \leq 0$.
The filtration naturally induces a filtration on the symmetric power $\Sym^k \calH$ for $k\geq1$.
In particular, for an integer $r \in [0,k]$, $\Fil^{k-r}\Sym^k\calH \cong \omega^{k-r} \otimes \Sym^r\calH$.

Recall that the relative de Rham cohomology $\calH$ admits a Gauss--Manin connection $\nabla: \calH \to \calH \otimes \Omega^1$, or more generally $\nabla^k: \Sym^k \calH \to \Sym^k\calH \otimes \Omega^1$.
For an integer $r \in [0, k]$, Griffiths transversality implies that the Gauss--Manin connection on $\Sym^k \calH$ induces a differential map
\[
\nabla^{k,r}:\
\omega^{k-r} \otimes \Sym^r\calH \cong \Fil^{k-r}\Sym^k\calH \xrightarrow{ \nabla^k}
\Fil^{k-r-1}\Sym^k\calH \otimes \Omega^1 \cong \omega^{k-r+1} \otimes \Sym^{r+1}\calH.
\]
When $k \geq 2r+2$, taking global sections over $X_{\Qp}$ gives rise to a differential operator
\[
\nabla^{k,r}: M_{k,r} \to M_{k+2, r+1}.
\]
We call it \emph{the classical Gauss--Manin connection}.

\subsection{Splitting of the Hodge filtration}
Our goal is to consider the variation of $\nabla^{k,r}$ as $k$ varies $p$-adically with $r$ \emph{fixed}.
For this, we first  choose a splitting of the Hodge filtration~\eqref{E:Hodge filtration}.
We will then show that two different such choices result in equivalent $p$-adic interpolations.\footnote{In fact, to $p$-adically interpolate Gauss--Manin connections, we need only choose the splitting of the Hodge filtration Zariski locally on $X$.  We choose a global splitting here to simplify the presentation.  However, it would be certainly interesting to know if one can entirely avoid the splitting of Hodge filtration in the following construction.}

A \emph{splitting} of the Hodge filtration is a homomorphism of coherent sheaves $\eta: \calH \to \omega$ which is a left inverse of the natural embedding.
This is equivalent to giving an embedding $\omega^{-1} \cong \calH/\omega \to \calH$ that is a right inverse of the natural projection.
Given a splitting $\eta$, one may write $\calH = \omega \oplus \omega^{-1}$.
Then, the Gauss--Manin connection on $\calH$ can be written as
\[
\begin{array}{rcccc}
	\nabla &:& \calH = \omega \oplus \omega^{-1} & \longrightarrow & \calH \otimes \Omega^1\cong \omega^3 \oplus \omega \\
	&&(x,y) & \longmapsto & (\nabla_1(x) + \nabla_2(y), \nabla_3(x) + \nabla_4(y)).
\end{array}
\]

We now analyse each component  $\nabla_i$ of $\nabla$.
\begin{itemize}
\item
We put $\partial = \nabla_1$, it is 
a connection on $\omega$ given by
\[
\partial: \omega \hookrightarrow \calH \xrightarrow{\nabla} \calH \otimes \Omega^1
\xrightarrow{\eta \otimes \mathrm{id}}  \omega \otimes \Omega^1.
\]

For $k \in \ZZ$, let $\partial^{\otimes k}$ denote the induced connection on the $k$-th power $\omega^k$. (Negative powers are understood as duals.)
Alternatively, using the Kodaira--Spencer isomorphism, we may view $\partial^{\otimes k}$ as a first-order differential operator $\omega^k \to \omega^{k+2}$. Our first step later on will be to $p$-adically interpolate $\partial^{\otimes k}$.

\item
For $\nabla_2$, we observe that the first term of the equality $y \otimes \nabla(a) + a\nabla(y)=\nabla(ay)$ lies in $\omega\cong\omega^{-1}\otimes\Omega^1\subseteq\calH \otimes \Omega^1$.
Hence, $\nabla_2$ is in fact $\calO_X$-linear and so is given by multiplication by a section $\lambda$ of $\omega^4$ over $X$.

\item
For $\nabla_3$, we note that the identification between $\Omega^1$ with $\omega^2$ is given by the Kodaira--Spencer isomorphism.  Thus, $\nabla_3$ is simply the identity map by definition.

\item
For $\nabla_4$, we observe that $\wedge^2 \calH$ is the trivial line bundle with the trivial connection.
This implies that $\nabla_4$ is the connection $\partial^{\otimes(-1)}$ on $\omega^{-1}$.

\end{itemize}

\begin{example}
\label{E:katz splitting}
An example of a splitting of the Hodge filtration was given by Katz \cite[A1.2]{katz}; he constructed a so-called ``canonical splitting'' using the equation for the universal elliptic curve.

For $k \in \ZZ_{\geq1}$, we put $\sigma_k(n) = \sum_{d|n} d^k$.
Using this splitting, the action of the differential operator on a form $f$ of weight $k$ is given in terms of $q$-expansions by
\[
\partial(f) = q\frac{df}{dq} + \frac{kE_2f}{12} \quad\text{and}\quad \lambda = E_4 = 1+ 120 \sum_{n\in \NN} \sigma_3(n)q^n.
\]
where $E_2 = 1 - 24 \sum_{n \in \ZZ_{\geq1}} \sigma_1(n) q^n$ and $q = e^{2\pi iz}$.
Note that $E_2$ is a $p$-adic modular form, but not an overconvergent one \cite{E2}.
The modular form $E_4$ is an Eisenstein series.
\end{example}

\begin{lemma}
\label{L:change of splitting}
Let $\eta'$ be another splitting of the Hodge filtration.  Then $\eta' - \eta$ induces a natural homomorphism of coherent sheaves $\calH / \omega \cong \omega^{-1} \to \omega$, which is given by multiplication by some section $\alpha$ of $\omega^2$.
Let $\partial'$ and $\lambda'$ be the associated differential operator and section of $\omega^4$ defined as above with respect to the splitting $\eta'$.
Then, we have
\[
\partial' = \partial + \alpha \quad \textrm{and} \quad \lambda' = \lambda - \alpha^2 - \partial^{\otimes 2}(\alpha).
\]
\end{lemma}
\begin{proof}
Looking at the change of basis matrix, we have
\[
\begin{pmatrix}
1 & \alpha \\ 0 & 1
\end{pmatrix}
\begin{pmatrix}
\partial & \lambda \\ 1 & \partial^{\otimes(-1)}
\end{pmatrix}
\begin{pmatrix}
1 & -\alpha \\ 0 & 1
\end{pmatrix}
 = \begin{pmatrix}
\partial' & \lambda' \\ 1 & \partial'^{\otimes(-1)}
\end{pmatrix}.
\]
From this, it is clear that $\partial' = \partial + \alpha$.
Multiplying out the matrix product, we have
\[
\lambda' = \lambda + \alpha \partial ^{\otimes(-1)} + (\partial + \alpha)\circ (-\alpha) = \lambda - \alpha^2 +  \alpha \partial ^{\otimes(-1)} - \partial \circ \alpha.
\]
Noting that $\partial \circ \alpha = \partial^{\otimes 2}(\alpha) + \alpha \partial^{\otimes(-1)}$,
the second formula of the lemma then follows.
\end{proof}

The following lemma is essentially \cite[A1.4]{katz}; it is a simple computation of tensors.

\begin{lemma}[Katz]
\label{L:katz}
For integers $k\geq1$ and $r \in [0, k]$, the chosen splitting $\eta$ induces an isomorphism $\omega^{k-r} \otimes \Sym^r \calH \cong \oplus_{a = 0}^{r} \omega^{k-2a}$.  Under this identification, the Gauss--Manin connection $\omega^{k-r}\otimes\Sym^r \calH \to \omega^{k-r+1} \otimes \Sym^{r+1}\calH$ is given by the sum over the maps
\[
\xymatrix@R=0pt@C=7pt{
\nabla^{k,r}:&
\omega^{k-2a} \ar[rr] && \omega^{k-2a} \ar@{}[r]|\bigoplus & \omega^{k-2a+2} \ar@{}[r]|\bigoplus &  \omega^{k-2a+4}\\
&
f \ar@{|->}[rr] && \big( \ (k-a)f, & \partial^{\otimes(k-2a)}(f), & a\lambda f \ \big).
}
\]
\end{lemma}

From this lemma, we see that the $p$-adic interpolation of the Gauss--Manin connection will essentially follow from the $p$-adic interpolation of the connections on ``$p$-adic powers" of $\omega$.

\subsection{Connection on $M^\dagger(\chi_A)$}
We will use the interplay between connections on vector bundles and connections on principal bundles to construct the $p$-adic interpolation of the Gauss--Manin connection.

Let $\pi_{T^\times}:T^\times \to X$ denote the natural projection.
Recall that, after choosing a splitting of the Hodge filtration, we have a connection $\partial: \omega \to \omega \otimes_{\calO_X} \Omega^1$.
This naturally induces a connection
\begin{equation}
\label{E:definition of tilde partial}
\xymatrix@R=5pt@C=5pt{
\quad\quad \tilde \partial: & \pi_{T^\times, *} \calO_{T^\times} \ar@{=}[d] \ar@{-->}[rrr] &&&\big(\pi_{T^\times, *} \calO_{T^\times}\big)\otimes _{\calO_X} \Omega^1 \ar@{=}[d]
\\
\displaystyle\bigoplus_{k \in \ZZ} \partial^{\otimes k}: & \displaystyle\bigoplus_{k \in \ZZ} \omega^k  \ar[rrr] & && \big(\displaystyle\bigoplus_{k \in \ZZ} \omega^k \big) \otimes _{\calO_X} \Omega^1.
}
\end{equation}
This may be viewed as a connection  on the $\GG_m$-torsor $T^\times$ over $X$.  Hence the connection is equivalent to a $\GG_m \times X$-equivariant $\calO_{T^\times}$-linear map
\begin{equation}
\label{E:connection-Gm-bundle}
\partial_{T^\times}:\ \Omega^1_{T^\times/\Zp}(\log \pi_{T^\times}^{-1}(C)) \longrightarrow \pi_{T^\times}^*\Omega^1
\end{equation}
which splits the natural map $\pi_{T^\times}^* \Omega^1 \hookrightarrow \Omega^1_{T^\times/\Zp}(\log \pi_{T^\times}^{-1}(C))$.

We may first take the analytification of $T^\times_{\QQ_p}$ as well as the map \eqref{E:connection-Gm-bundle} (tensored with $\QQ_p$), and then restrict the map to $F_n^\times$.  We thus obtain an $\calO_{G_n}$-equivariant (because $G_n$ is a subgroup of $\GG_{m, \rig} \times X(v)$) $\calO_{F^\times_n}$-linear map
\[
\partial_{F_n^\times}:\
\Omega_{F^\times_n/\Qp}^1(\log \pr_n^{-1}(C)) \to \pr_n^*\Omega^1_{X(v)/\QQ_p}(\log C)
\]
which splits the natural map $\pr_n^*\Omega^1_{X(v)/\QQ_p}(\log C) \hookrightarrow \Omega_{F^\times_n/\Qp}^1(\log \pr_n^{-1}(C))$.
This gives the connection on the $G_n \times X(v)$-torsor $F^\times_n$ over $X(v)$.  In particular, 
pre-composing $\partial_{F^\times_n}$ with the natural map $\calO_{F^\times_n} \xrightarrow{d} \Omega_{F^\times_n/\Qp}^1(\log \pr_n^{-1}(C))$ and then pushing forward along $\pr_n$ induces a natural $\calO_{G_n}$-equivariant map
\begin{equation}
\label{E:connection-G_n-bundle}
\tilde \partial_{F_n^\times}:\
\pr_{n, *}\calO_{F^\times_n} \longrightarrow
\pr_{n, *}\calO_{F^\times_n} \otimes_{\calO_{X(v)}} \Omega^1_{X(v)/\QQ_p}(\log C).
\end{equation}

Now, let $A$ be an affinoid algebra and $\chi_A:\ZZ_p^\times \to A^\times$ a locally analytic character as in \S~\ref{S:locally analytic characters}.  We may tensor \eqref{E:connection-G_n-bundle} with $\calV$ and then take the $G_n$-invariant sections.  This gives a natural connection
\[
\begin{array}{lcccll}
	\partial^{\chi_A} &:&\boldsymbol \omega^{\chi_A}&=&\multicolumn{2}{l}{\big((\mathrm{pr}_n \times \id)_* \calO_{F_n^\times \times \Spm A} \otimes_{\calO_{X(v) \times \Spm A}} \calV \big)^{G_n}} \\
					&&&& \xrightarrow{\tilde \partial_{F_n^\times}}& \big((\mathrm{pr}_n \times \id)_* \calO_{F_n^\times \times
						\Spm A} \otimes_{\calO_{X(v)}} \Omega^1_{X(v)/\QQ_p}(\log C)\otimes_{\calO_{X(v) \times \Spm A}} \calV \big)^{G_n} \\
					&&&&& \quad = \boldsymbol \omega^{\chi_A} \otimes_{\calO_{X(v)}} \Omega^1_{X(v)/\QQ_p}(\log C).
\end{array}
\]

Taking global sections over $X(v) \times\Spm(A)$ induces a natural map
\[
\partial^{\chi_A}: M^{\dagger, v}(\chi_A) \to M^{\dagger, v}(\chi_A) \otimes_{\Gamma(X(v), \calO_{X(v)})} \Gamma(X(v), \Omega^1_{X(v)/\QQ_p}(\log C)) \cong M^{\dagger, v}(\chi_A \chi_\cycl^2),
\]
and $\partial^{\chi_A}: M^\dagger(\chi_A) \to M^\dagger(\chi_A\chi_\cycl^2)$.

When the character $\chi_A$ is $\chi_\cycl^k$, it is easy to see that $\partial^{\chi_\cycl^k}$ is the restriction of $\partial^{\otimes k}$ on $\omega^k$ to $X(v)$.

\begin{lemma}
\label{L:change of splitting family}
If we change the splitting of Hodge filtration as in Lemma~\ref{L:change of splitting}, then the corresponding family of connections $\partial^{\chi_A}$ and $\partial'^{\chi_A}$ are related by
\[
\partial'^{\chi_A} = \partial^{\chi_A} + \WT(\chi_A) \cdot \alpha.
\]
\end{lemma}
\begin{proof}
This can be proved by specializing to classical characters as in Proposition~\ref{P:independence of splitting} below.
But we prefer a down-to-earth proof which we hope will give the reader a better intuition for the construction.

Let $\eta'$ be another splitting of Hodge filtration and let $\alpha$ be as in Lemma~\ref{L:change of splitting}.
Let $\partial'$, $\tilde \partial'$, $\partial'_{T^\times}$, $\partial'_{F^\times_n}$, and $\tilde \partial'_{F^\times_n}$ be the differential operators or connections constructed using $\eta'$ in place of $\eta$.
We analyse their difference with the original one.

By Lemma~\ref{L:change of splitting}, $\partial' -\partial = \alpha$ can be viewed as a section of $\omega^2 \cong \Omega^1$.
Then $\tilde \partial' - \tilde \partial$ is multiplication by $k \alpha$ on the direct summand $\omega^k$.
In other words, the action of $\tilde \partial' - \tilde \partial$ is given by the action of the Lie algebra of $\GG_m$ multiplied with $\alpha$.
It is clear that this property continues to hold for the differences of the other differential operators.  In particular, $\partial'^{\chi_A} - \partial^{\chi_A}$ is given by the action of the Lie algebra multiplied by $\alpha$.  This is the statement of the lemma.
\end{proof}

\subsection{The family of Gauss--Manin connections on $M^\dagger_r(\chi_A)$}
\label{S:family GM connection}

Keep $\chi_A$ and $v$ as above.
Recall that the character $\chi_A$ gives rise to a weight function $\WT = \WT(\chi_A)$ on $\Spm(A)$.
Following Lemma~\ref{L:katz},
we define the \emph{family of nearly overconvergent Gauss--Manin connections} to be
\[
\nabla^{\chi_A, r}: M^{\dagger, v}_r (\chi_A) = 
\bigoplus_{a = 0}^r M^{\dagger,v}(\chi_A \chi_\cycl^{-2a}) \longrightarrow 
M^{\dagger, v}_{r+1} (\chi_A\chi_\cycl^2) = 
\bigoplus_{a = 0}^{r+1} M^{\dagger,v}(\chi_A \chi_\cycl^{2-2a})
\]
given by sending $f \in M^{\dagger,v}(\chi_A \chi_\cycl^{-2a})$ to
\begin{equation}
\label{E:definition of nabla^chi_A r}
(\ (\WT - a)f,\ \partial^{\chi_A\chi_\cycl^{-2a}}(f), \
a \lambda f\ ) \in M^{\dagger,v}(\chi_A \chi_\cycl^{-2a}) \oplus M^{\dagger,v}(\chi_A \chi_\cycl^{2-2a}) \oplus M^{\dagger,v}(\chi_A \chi_\cycl^{4-2a}).
\end{equation}
Taking $v \to 0^+$, this defines $\nabla^{\chi_A, r}: M^{\dagger}_r (\chi_A)
\to M^{\dagger}_{r+1} (\chi_A\chi_\cycl^2).
$
When $\chi_A = \epsilon\chi_\cycl^k$ is a classical character and $k \geq 2r+2$, the Gauss--Manin connection  $M^\dagger_{k,r}(N, \epsilon_N) \to M^\dagger_{k+2, r+1}(N, \epsilon_N)$ is compatible with the algebraic Gauss--Manin connection $\nabla_{k,r}: M_{k,r}(N, \epsilon_N) \to M_{k+2, r+1}(N, \epsilon_N)$.

\begin{proposition}
\label{P:independence of splitting}
The family of nearly overconvergent Gauss--Manin connections $\nabla^{\chi_A, r}$ defined above does not depend on the choice of the splitting $\eta$ of the Hodge filtration.
\end{proposition}
\begin{proof}
We can check this by hand using Lemmas~\ref{L:change of splitting} and \ref{L:change of splitting family} and the expression \eqref{E:definition of nabla^chi_A r}.
This amounts to checking a matrix equality.
We leave the details to the interested reader.

Alternatively, we can check the independence using an abstract argument as follows.
Note that when $\chi_A = \chi_\cycl^k$ for an integer $k \geq 2r+2$, $\nabla^{\chi_A, r}$ is the same as the Gauss--Manin connection $\nabla^{k,r}$ and hence is independent of the choice of splitting of the Hodge filtration.
In general, using the functoriality of the construction, we may assume that $\Spm(A) $ is a geometrically connected subdomain of $\calW$ containing infinitely many characters of the form above.
Since $M^{\dagger,v}_r(\chi_A)$ is potentially orthogonalizable in the sense of Buzzard \cite{buzzard}, the operator $\nabla^{\chi_A, r}$ is determined by its specializations to these classical characters. Hence $\nabla^{\chi_A,r}$ is independent of the choice of splitting of the Hodge filtration.
\end{proof}

\begin{remark}
It is clear from the construction that the family of nearly overconvergent Gauss--Manin connections $\nabla^{\chi_A,r}$ commutes with the action of Hecke operators.
\end{remark}

\subsection{$q$-expansions}
Neither \cite{AIS} nor \cite{pilloni} elaborated on the $q$-expansion attached to overconvergent modular forms in an explicit way.  We include a short discussion for completeness.

Consider the Tate curve $E=\Tate(q)$ over $\ZZ_p((q))$; it comes equipped with a canonical differential $\omega_\can = \frac{dt}{t}$ and its d.s.k ``dual" $\eta_\can$ as defined in \cite[A1.3.14]{katz}.

The Tate curve admits a canonical subgroup $C_n = \mu_{p^n} \subseteq E[p^n]$.  Its dual is canonically isomorphic to $\ZZ/p^n\ZZ$ over $\Zp\llbracket q\rrbracket$.  The Hodge--Tate map for the Tate curve is 
\[
\HT:
C_n^D \cong (\ZZ/p^n\ZZ)_{\Zp((q))} \to \Omega^1_{C_n/ \Zp((q))} \cong \Omega^1_{E/\Zp((q))} \otimes_{\calO_E} \calO_{C_n}
= \calO_{C_n}\cdot\omega_\can.
\]
It sends the canonical generator $1\in \ZZ/p^n\ZZ$ to $\omega_\can$.
Moreover, this isomorphism extends to an isomorphism over $\Zp\llbracket q\rrbracket$.
Let $R= \Qp\langle p^{-1} q\rangle$ be the ring of analytic functions on the disk at a cusp of radius $p^{-1}$;\footnote{We could have worked with a larger disk or even with $\Zp\llbracket q\rrbracket [\frac 1p]$, but the latter does not fit into the language of rigid analytic space.} so we view $\Spm(R)$ as an affinoid subdomain of $X(v)$ for any $v$.  We put $z = \omega_\can$ so that $T^\times  \times_X \Spec R = \Spec R[z, z^{-1}]$.
The description of the Hodge--Tate map implies that
\[
F^\times_n \times_{X(v)} \Spm R \cong \coprod_{m \in\ZZ/p^n\ZZ} B_R(x_m, p^{-n}),
\]
where $B_R(x_m, p^{-n})$ denotes the relative disk over $R$ around $z=x_m$ of radius $p^{-n}$.

Now, let $\chi_A: \ZZ_p^\times \to A^\times$ be a continuous character.  Then for some $n \geq 3$, we have $|\chi_A(\exp(p^{n-1})) - 1|< p^{-1/(p-1)}$.  We put $\lambda = \chi_A(\exp(p^n)) \in A^\times$.  It follows that we can construct the overconvergent sheaf using this $n$.  Let $G_{n,R}^\circ$ denote the identity component of $G_{n,R}: = G_n \times_{X(v)} \Spm(R)$ and let $F^{\times, \circ}_{n,R}$ be the connected component of $F^\times_{n,R}: = F^\times_n \times_{X(v)} \Spm(R)$ that contains $\omega_\can$. 
We consider the following section of $\calV \otimes_{\calO_{X(v) \times \Spm A}}\calO_{F^{\times, \circ}_{n,R}}$
\[
z^{\log(\lambda)/ p^n} = 1+ \sum_{i = 1}^\infty (z-1)^{i} \otimes \binom{\log(\lambda)/ p^n}i.
\]
It is $G_{n,R}^\circ$-equivariant in the sense that, for $g \in 1+p^n \Zp\llbracket q\rrbracket$,
\[
g(z^{\log(\lambda)/ p^n}) = (g^{-1}z)^{\log(\lambda)/ p^n}\chi(g) = z^{\log(\lambda)/p^n} \chi(g) g^{-\log(\lambda)/p^n} = z^{\log(\lambda)/p^n}.
\]

Using the action of $G_{n,R}$, we can extend this section to a $G_{n,R}$-equivariant section of $\calV \otimes_{X(v) \times \Spm(A)}\calO_{F^{\times}_{n,R}}$.  In other words, we have constructed an explicit section $\omega_\can^{\chi_A}$ of $\boldsymbol \omega^{\chi_A}$ over $\Spm(R \widehat \otimes A)$.

Thus, for any nearly overconvergent modular form $f \in M^{\dagger,v}_r(\chi_A)$, we may evaluate $f$ on this Tate curve and write 
\[
f = \bff_0(q) \omega_\can^{\chi_A} \omega_\can^r +\bff_1(q) \omega_\can^{\chi_A} \omega_\can^{r-1}\eta_\can + \cdots +\bff_r(q) \omega_\can^{\chi_A} \eta_\can^r
\]
 for some $\bff_0, \dots, \bff_r \in A\llbracket q\rrbracket$.   This gives a natural morphism
\begin{align}
\label{E:q-expansion map}
M^{\dagger, v}_r(\chi_A) &\longrightarrow A\llbracket q\rrbracket[Y]
\\
\nonumber
f & \longmapsto \bff_0(q) + \bff_1(q)Y + \cdots + \bff_r Y^r.
\end{align}
functorial in the character $\chi_A$. When $\chi_A$ is a classical character and $r=0$, the map $f \mapsto \bff_0(q)$ recovers the $q$-expansion map of overconvergent modular forms.  
When the modular curve $X$ is geometrically connected, as in the proof of usual $q$-expansion principle, $\Spec(R)$ has Zariski dense image in $X_{\QQ_p}$; so passing to the $q$-expansion \eqref{E:q-expansion map} is equivalent to taking completion at the corresponding cusps and is hence injective.
Thus, one often uses the $q$-expansion to indicate the corresponding nearly overconvergent modular form.

\begin{example}
Assume that we choose the splitting of the Hodge filtration to be the one given by Katz as in Example~\ref{E:katz splitting}.
In terms of $q$-expansions, $\partial^{\chi_A}: M^\dagger(\chi_A) \to M^\dagger(\chi_A\chi_\cycl^2)$ is then given by
\[
\bff \mapsto q\frac d{dq}\bff + \frac {\WT(\chi_A)E_2 \bff}{12}.
\]
\end{example}

\vspace{10pt}
\noindent{\footnotesize (Robert Harron) \textsc{Department of Mathematics, Van Vleck Hall, University of Wisconsin--Madison, Madison, WI 53706, USA.} \textit{As of August 2014:} \textsc{Department of Mathematics, Keller Hall, University of Hawai`i at M\={a}noa, Honolulu, HI 96822, USA}\\
\textit{Email address:} \texttt{rharron{@}hawaii.edu}
\\
\\
{\footnotesize (Liang Xiao) \textsc{Department of Mathematics, 340 Rowland Hall, University of California, Irvine, Irvine, CA 92697, USA.} \textit{As of August 2014:} \textsc{Department of Mathematics, Mathematical Sciences Building, University of Connecticut, Storrs, Storrs, CT 06269, USA}\\
\textit{Email address:} \texttt{liang.xiao{@}uconn.edu}

\begin{thebibliography}{9999}

\bibitem[AIP]{AIP}
F. Andreatta, A. Iovita, and V. Pilloni,
$p$-adic families of Siegel modular cuspforms, {\it to appear in Ann.\ of Math.}

\bibitem[AIS]{AIS}
F. Andreatta, A. Iovita, and G. Stevens, On overconvergent modular sheaves and modular forms for $GL_{2/F}$, preprint, {\it to appear in Israel J.\ Math.}, doi:10.1007/s11856-014-1045-8.

\bibitem[B]{buzzard}
K.~Buzzard, Eigenvarieties, in  \textit{$L$-functions and Galois representations}, pp.~59--120, \textit{London Math. Soc. Lecture Note Ser.}, \textbf{320}, Cambridge Univ. Press, Cambridge, 2007.

\bibitem[CGJ]{E2}
R.~Coleman, F.~Gouv\^ ea, and N.~Jochnowitz, $E_2$, $\Theta$, and overconvergence. 
{\it Int.\ Math.\ Res.\ Not.} {\bf 1995}, no.~1, pp.~23--41.

\bibitem[DR]{DarmonRotger}
H.~Darmon and V.~Rotger, Diagonal cycles and Euler systems I: A $p$-adic Gross--Zagier formula, {\it to appear in Ann.\ Sci.\ \'{E}c.\ Norm.\ Sup\'{e}r.\ (4)}.

\bibitem[F]{fargues}
L. Fargues, La filtration de Harder--Narasimhan des sch\'emas en groupes finis et plats, {\it J. Reine Angew. Math.} {\bf 645} (2010), 1--39.

\bibitem[K]{katz}
N.~Katz, $p$-adic properties of modular schemes and modular forms, in {\it Modular functions of one variable, III}, pp. 69--190. {\it Lecture Notes in Mathematics}, Vol. {\bf 350}, Springer, Berlin, 1973.

\bibitem[P]{pilloni}
V.~Pilloni, Overconvergent modular forms, \textit{Ann. Inst. Fourier} \textbf{63} (2013), no. 1, pp.~219--239.

\bibitem[U]{urban}
E.~Urban, Nearly overconvergent modular forms, {\it  to appear in the Proceedings of conference IWASAWA 2012 held at Heidelberg}, available at 
\texttt{http://www.math.columbia.edu/\textasciitilde{}urban/EURP.html}.

\bibitem[U2]{Urban2}
E.~Urban, On the rank of Selmer groups for elliptic curves over $\QQ$, {\it Proceedings of the International Colloquium TIFR (2013)}.
\end{thebibliography}
\end{document}